\documentclass[oneside]{amsart}
\usepackage{geometry}                % See geometry.pdf to learn the layout options. There are lots.
\usepackage{graphicx}
\usepackage{amsmath}
\usepackage{amssymb}
\usepackage{amsthm}
\usepackage{mathtools}
\usepackage{mathscinet}
\usepackage{tikz-cd}
\usepackage{dsfont}
\usepackage{color}
\usepackage{enumitem}
\usepackage{tensor}
\usepackage{nicefrac}
\usepackage{todonotes}
\usepackage{mathrsfs}
\usepackage{float}

\usepackage[]{hyperref}
\hypersetup{
	pdftitle={Non-vanishing of Class Group L-functions for Number Fields with a Small Regulator},
	pdfauthor={Ilya Khayutin},
	bookmarksnumbered=true,     
	bookmarksopen=true,         
	bookmarksopenlevel=1,  
	bookmarksdepth=3,     
	colorlinks=true,
	citecolor=blue,            
	pdfstartview=Fit,           
	pdfpagemode=UseOutlines,   
	pdfpagelayout=TwoPageRight
}

\usepackage[protrusion=true]{microtype}
\DeclareSymbolFont{legacymaths}{OT1}{cmr}{m}{n}
\SetSymbolFont{legacymaths}{bold}{OT1}{cmr}{bx}{n}
\usepackage[vvarbb,upint]{newtxmath}
\usepackage[cal=boondoxo]{mathalfa} 

\makeatletter
\let\orgdescriptionlabel\descriptionlabel
\renewcommand*{\descriptionlabel}[1]{%
  \let\orglabel\label
  \let\label\@gobble
  \phantomsection
  \protected@edef\@currentlabel{#1}%
  \let\label\orglabel
  \orgdescriptionlabel{#1}%
}
\def\paragraph{
	\@startsection{paragraph}{4}
	\z@{.5\linespacing\@plus.7\linespacing}{-.5em}%
	{\normalfont\itshape}}
\makeatother
\DeclareFontFamily{U}{mathx}{\hyphenchar\font45}
\DeclareFontShape{U}{mathx}{m}{n}{
      <5> <6> <7> <8> <9> <10>
      <10.95> <12> <14.4> <17.28> <20.74> <24.88>
      mathx10
      }{}
\DeclareSymbolFont{mathx}{U}{mathx}{m}{n}
\DeclareFontSubstitution{U}{mathx}{m}{n}
\DeclareMathAccent{\widecheck}{0}{mathx}{"71}
\newcommand{\dif}{\,\mathrm{d}}

\DeclareMathOperator{\disc}{Disc}
\DeclareMathOperator{\Reg}{Reg}
\DeclareMathOperator{\Nr}{Nr}

\DeclareMathOperator{\diag}{diag}

\DeclareMathOperator{\Res}{Res}

\newcommand{\Cl}{\mathrm{Cl}}

\renewcommand{\det}{\operatorname{det}} %align indices correctly in equation mode

\newlength{\faktorheight}
\newcommand*{\dfaktor}[3]{% \dfaktor{#1}{#2}{#3} -> #1\#2/#3
\mathchoice{
 \settototalheight{\faktorheight}{\ensuremath{#1}}%
  \raisebox{-0.5\faktorheight}{\ensuremath{#1}}% Denominator left
   \backslash
   \settototalheight{\faktorheight}{\ensuremath{#2}}%
  \raisebox{0.5\faktorheight}{\ensuremath{#2}}% Numerator
  \slash
   \settototalheight{\faktorheight}{\ensuremath{#3}}%
  \raisebox{-0.5\faktorheight}{\ensuremath{#3}}% Denominator right
  }{
  \ensuremath{#1}
  \backslash
  \ensuremath{#2}
  \slash
  \ensuremath{#3}
  }
  {\let\GenerateWarning=\errmessage}
  {\let\GenerateWarning=\errmessage}
}

\newcommand*{\lfaktor}[2]{% lfaktor{#1}{#2} -> #1\#2
\mathchoice{
 \settototalheight{\faktorheight}{\ensuremath{#1}}%
  \raisebox{-0.5\faktorheight}{\ensuremath{#1}}% Denominator
	\backslash
   \settototalheight{\faktorheight}{\ensuremath{#2}}%
  \raisebox{0.5\faktorheight}{\ensuremath{#2}}%Numerator 
}
{
	\ensuremath{#1}
	\backslash
	\ensuremath{#2}
}
{\let\GenerateWarning=\errmessage}
{\let\GenerateWarning=\errmessage}
}

\newcommand*{\faktor}[2]{% faktor{#1}{#2} -> #1/#2
\mathchoice{
 \settototalheight{\faktorheight}{\ensuremath{#1}}%
  \raisebox{0.5\faktorheight}{\ensuremath{#1}}%Numerator 
  \slash
   \settototalheight{\faktorheight}{\ensuremath{#2}}%
  \raisebox{-0.5\faktorheight}{\ensuremath{#2}}% Denominator
}
{
	\ensuremath{#1}
	\slash
	\ensuremath{#2}
}
{\let\GenerateWarning=\errmessage}
{\let\GenerateWarning=\errmessage}
}

\newtheorem{thm}{Theorem}
\numberwithin{thm}{section} 
\newtheorem*{thm*}{Statement of Theorem}
\newtheorem*{thm-quote}{Theorem}
\newtheorem{lem}[thm]{Lemma}
\newtheorem{prop}[thm]{Proposition}
\newtheorem{cor}[thm]{Corollary}

\theoremstyle{definition}
\newtheorem{defi}[thm]{Definition}

\theoremstyle{remark}
\newtheorem{remark}[thm]{Remark}

\title[Non-vanishing of Class Group $L$-functions]{Non-vanishing of Class Group $L$-functions for Number Fields with a Small Regulator}
\author[I. Khayutin]{Ilya Khayutin}
\address{Mathematics Department, Northwestern University, 2033 Sheridan Road
Evanston, IL 60208}
\begin{document}

%\date{}
\begin{abstract}
Let $E/\mathbb{Q}$ be a number field of degree $n$. We show that if $\Reg(E)\ll_n |\disc(E)|^{1/4}$ then the fraction of class group characters for which the Hecke $L$-function does not vanish at the central point is $\gg_{n,\varepsilon} |\disc(E)|^{-1/4-\varepsilon}$. 

The proof is an interplay between almost equidistribution of Eisenstein periods over the toral packet in $\mathbf{PGL}_n(\mathbb{Z})\backslash\mathbf{PGL}_n(\mathbb{R})$ associated to the maximal order of $E$, and the escape of mass of the torus orbit associated to the trivial ideal class. 
\end{abstract}
\maketitle

\section{Introduction}
Our main result is the following theorem.
\begin{thm}\label{thm:main}
Let $E/\mathbb{Q}$ be a number field of degree $n$. Denote by $D$ its discriminant, by $R$ the regulator of its ring of integers and by $h$ the class number. 
For every class group character $\chi\in \widehat{\Cl(E)}$ let $L(s,\chi)$ be the associated Hecke $L$-function.

Fix a real number $1/2\leq s<1$. There are effectively computable constants $A,B>0$ that depends only on $s,n$ such that for every $1/2>\varepsilon>0$
\begin{equation*}
h^{-1}\#\left\{\chi \in \widehat{\Cl(E)} \mid L(s,\chi)\neq 0 \right\} 
\geq |D|^{-(1-s+\varepsilon)/2} \left(A-B \frac{R}{|D|^{s/2}}\right) \varepsilon^n\;.
\end{equation*}
\end{thm}
The most interesting point is of course $s=1/2$ as GRH would imply non-vanishing of $L(s,\chi)$ at $1/2<s<1$ for all $\chi$. Fr\"ohlich \cite{Frohlich} has demonstrated that the Dedekind zeta function actually vanishes at the central point for infinitely many number fields. Duke \cite{DukeLarge} has constructed for each $n$ an infinite family of degree $n$ totally real $S_n$ number fields such that $R\ll (\log |D|)^{(n-1)}$.

There is a very rich literature about non-vanishing of $L$-functions at the central point for several families of $L$-functions. In this exposition we restrict our discussion to class group $L$-functions and closely related families. Blomer \cite{Blomer} has established a very strong result for the family of class group $L$-functions of imaginary quadratic fields. He is able to demonstrate non-vanishing for a large fraction of the class group characters, $\gg \varphi(|D|)/|D|$, whenever $|D|\gg 1$. Theorem \ref{thm:main} provides significantly weaker results for imaginary quadratic fields but it covers class group $L$-functions of any degree. In the conductor aspect, Balasubramanian and Murty \cite{BalasubramanianMurty} established that a positive proportion of Dirichlet $L$-function of prime conductor $q\gg 1$ do not vanish at the central point.
Soundararajan \cite{SoundNonVanishing} has established that a positive proportion of Dedekind zeta functions of real quadratic fields do not vanish at the central point. Methodologically, the work of Michel and Venkatesh \cite{MichelVenkatesh} about non-vanishing of twists of automorphic $\mathbf{GL}_2$ $L$-functions by quadratic class group characters is the closest to ours.

We remark that predictions about the behavior of $L$-functions at the central point can often be deduced from random matrix theory heuristics \cite{KatzSarnak,SarnakShinTemplier,ShankarSodergrenTemplier}. Moreover, the non-vanishing phenomena is related to deep questions in analytic number theory, such as the existence of Landau-Siegel zeros \cite{IwaniecSarnak} and spectral gap for automorphic representations \cite{LuoRudnickSarnakI,LuoRudnickSarnakII}.

Three aspects of Theorem \ref{thm:main} stand out. The first is that the result is valid for number fields of any degree. The second is that we allow relatively large regulators. In particular, whenever $R=o(|D|^{1/4})$ Theorem \ref{thm:main} provides new non-vanishing results at the central point $s=1/2$. Finally, the non-vanishing fraction depends only on the discriminant and the regulator, and does not depend on the shape of the unit lattice. Specifically, we do not need to assume that the number field $E$ has no non-trivial subfields of a small regulator. The latter assumption is needed in the course of the proof of \cite[Theorem 1.10]{ELMVPeriodic} which is conceptually related to our method. Finally, it is worth mentioning that the constants $A,B$ are completely effective, and do not depend on Siegel's bound, cf.\ \cite{Blomer} where the lower bound for $|D|$ is ineffective. 
 
\subsection{Subconvexity}
Some improvements of the lower bound in Theorem \ref{thm:main} for $s=1/2$ are easily achievable.
\begin{enumerate}
\item Using the weak subconvexity bound of Soundararajan \cite{Sound} we can deduce a lower bound with a logarithmic improvement for all number fields
\begin{equation*}
h^{-1}\#\left\{\chi \in \widehat{\Cl(E)} \mid L(1/2,\chi)\neq 0 \right\} \gg_{n,\varepsilon} |D|^{-1/4}(\log |D|)^{1-\varepsilon}\left(A-B \frac{R}{|D|^{1/4}}\right)\;. 
\end{equation*}

\item Whenever there is $\delta>0$ such that a subconvex bound in the discriminant aspect 
\begin{equation*}
|L(1/2,\chi)| \ll_{n,\varepsilon} |D|^{(1/2-\delta+\varepsilon)/2}
\end{equation*}
is known, 
we can improve the lower bound to
\begin{equation*}
h^{-1}\#\left\{\chi \in \widehat{\Cl(E)} \mid L(1/2,\chi)\neq 0 \right\} \gg_{n,\varepsilon} |D|^{-(1/2-\delta+\varepsilon)/2} \left(A-B \frac{R}{|D|^{1/4}}\right)\;. 
\end{equation*}
The Grand Lindel\"of Hypothesis would provide the optimal $\delta=1/2$. A non-trivial $\delta>0$ is known unconditionally for abelian fields $E$ using the Burgess bound \cite{BurgessII} and for cubic fields using the convexity breaking results of Duke, Friedlander and Iwaniec \cite{DFI} and Blomer, Harcos, Michel \cite{BHM}.
\end{enumerate}

\subsection{Method of Proof}
To study the Hecke $L$-function at the critical strip we follow Hecke's original method \cite{Hecke}. That is, we represent the $L$-function as an integral of a spherical degenerate Eisenstein series $E(\bullet,s)\colon\lfaktor{\mathbf{PGL}_n(\mathbb{Z})}{\mathbf{PGL}_n(\mathbb{R})}\to\mathbb{C}$ along a collection of periodic torus orbits. This spherical Eisenstein series coincides with the Epstein zeta function of the associated quadratic form. The definition and properties of the Epstein zeta function are reviewed in \S\ref{sec:Epstein-cusp}.

Our strategy is most closely related to the methods of Michel and Venkatesh \cite{MichelVenkatesh} who study non-vanishing at the central point for twists of $\mathbf{GL}_2$ automorphic $L$-functions by quadratic class group characters. They provide two tools to establish non-vanishing in a family, either using effective equidistribution of a packet of Heegner points on the modular curve, or using the escape of mass of a portion of the packet that contains the trivial ideal class. Unfortunately, in higher rank we do not know unconditionally an effective equidistribution result of the analogues toral packets nor do we know that a large enough portion of the mass escapes to infinity, even for small regulators. Instead, we observe that combining very weak versions of both statement \emph{together} is sufficient to establishing the non-vanishing theorem. The equidistribution statement is weakened to the convexity bound of Hecke $L$-functions. It is supplemented with a good control on the mass that the single orbit of the trivial ideal class element spends high in the cusp. 

In section \S \ref{sec:periods} we construct a maximal torus $H<\mathbf{PGL}_n(\mathbb{R})$ from a fixed 
degree $n$ number field $E$ and an algebra isomorphism $\iota\colon E\otimes \mathbb{R}\to \mathbb{R}^{r_1}\times \mathbb{C}^{r_2}=\mathbb{R}^n$.
Every fractional ideal $\Lambda \subset E$ gives rise to a periodic $H$-orbit which we denote by $\tensor[^\iota]{\Lambda}{}H\subset \lfaktor{\mathbf{PGL}_n(\mathbb{Z})}{\mathbf{PGL}_n(\mathbb{R})}$, cf.\ \cite{ELMVPeriodic}. This periodic orbit depends only on the ideal class of $\Lambda$. We recall these classical definitions as well in \S\ref{sec:periods}. Fix $1/2\leq s < 1$ and define the function $Z\colon \Cl(E)\to\mathbb{C}$
\begin{align*}
Z(\Lambda)&\coloneqq\int_{[H]} E^*(\tensor[^\iota]{\Lambda}{}h,ns) \dif^\times h\;,\\
E^*(g,s)&\coloneqq \pi^{-s/2}\Gamma\left(\frac{s}{2}\right) E(g,s)\;,\\
E\left(g,s\right)&\coloneqq\frac{1}{2}|\det g|^{s/n} \sum_{0\neq \mathrm{v}\in \mathbb{Z}^n} \|\mathrm{v} g \|_2^{-s}\;.
\end{align*} 
The function $E\left(g,s\right)$ is the Epstein zeta function associated to the lattice $\mathbb{Z}^n g$ and
$E^*(g,\rho)$ is the completed Epstein zeta function. The integration is with respect to the $H$-periodic measure of volume $1$. Hecke's period formula, cf.\ Theorem \ref{thm:Hecke-period}, expresses this integral in terms of a completed partial Dedekind zeta function
\begin{equation*}
Z(\Lambda)=\frac{w}{2^{r_1} n R} \zeta^*_{\Lambda}(s)
\end{equation*}
whose definition we recall in \ref{defi:partial-dedekind}. The Fourier coefficient of this function coincides with the completed $L$-function of the class group character
\begin{equation*}
\hat{Z}(\chi)=\frac{w}{2^{r_1} n h R} L^*(s,\chi)
\end{equation*}
for any $\chi\in\widehat{\Cl(E)}$.

In Theorem \ref{thm:Epstein-cusp} we establish a good lower bound on the Epstein zeta function high in the cusp using an approximate functional equation. This lower bound and the fact that the lattice $\iota(\mathcal{O}_E)\subset \mathbb{R}^n$ contains the short vector $(1,\ldots,1)$ are used in the proof of the key statement of this manuscript -- Proposition \ref{prop:Z(O)-bound}. This proposition states that there are effectively computable constants $A_1,B_0>0$ such that
\begin{equation*}
Z(\mathcal{O}_E) \geq \frac{A_1|D|^{s/2}-B_0 R}{R}\;.
\end{equation*}
The proof of this result also uses a trick where the unit lattice is approximated by the lattice spanned by vectors realizing its successive minima. This allows us to remove the dependence on the shape of the unit lattice. V.\ Blomer has later suggested to the author a briefer proof of Proposition \ref{prop:Z(O)-bound} by applying the approximate functional equation directly to the partial Dedekind zeta function. The proof presented here emphasizes the role of the crucial concept of escape of mass.

Without further ado we establish Theorem \eqref{thm:main} assuming this result and using the following elementary lemma.  
\begin{lem}\label{lem:non-vanishing}
Let $C$ be a finite abelian group. For every function $f\colon C\to\mathbb{C}$ define 
\begin{equation*}
\operatorname{NV}(f)\coloneqq\left\{\chi \in \widehat{C} \mid  \hat{f}(\chi)\neq 0 \right\}\;.
\end{equation*}
Then
\begin{equation*}
 \#\operatorname{NV}(f) \geq \frac{\|f\|_\infty}{\|\hat{f}\|_\infty}\;.
\end{equation*}
\end{lem}
\begin{proof}
Fix $c\in C$ where $|f|$ attains its maximum. Then
\begin{equation*}
\|f\|_\infty=\left|\sum_{\chi\in\operatorname{NV}(f)} \hat{f}(\chi) \chi(c)\right|\leq \|\hat{f}\|_\infty \left|\operatorname{NV}(f)\right| 
\end{equation*}
\end{proof}

\begin{proof}[Proof of Theorem \ref{thm:main}]
We apply Lemma \ref{lem:non-vanishing} above to the function $Z\colon \Cl(E)\to\mathbb{C}$. Proposition \ref{prop:Z(O)-bound} provides the necessary lower bound on $\|Z\|_\infty$. We need only an appropriate upper-bound on $|\hat{Z}(\chi)|$ for any class-group character $\chi$. Recall the convexity bound for Hecke $L$-functions of class group characters, cf.\ \cite{Rademacher}. For every $0<\varepsilon<1/2$ and $1/2\leq s < 1$
\begin{equation*}
|L(s,\chi)|\ll_{s,n} |D|^{(1-s+\varepsilon)/2} \zeta(1+\varepsilon)^n\Longrightarrow |\hat{Z}(\chi)|\ll_{s,n} \frac{|D|^{(1+\varepsilon)/2}}{hR}\zeta(1+\varepsilon)^n \ll_n \frac{|D|^{(1+\varepsilon)/2}}{hR} \varepsilon^{-n}
\end{equation*}
Dividing the lower bound from Proposition \ref{prop:Z(O)-bound} by the convexity upper bound implies the claimed theorem. 
\end{proof}
The convexity bound should be understood as an almost-equidistribution statement for periods of degenerate Eisenstein series over toral packets. Indeed, any subconvex improvement in the discriminant aspect over the convexity bound would imply the equidistribution of any degenerate pseudo-Eisenstein series. If $n$ is a prime then such a subconvexity bound can be bootstrapped using the method of Einsiedler, Lindenstrauss, Michel and Venkatesh \cite{ELMVCubic} to equidistribution of any compactly supported continuous function.

\subsection*{Acknowledgments}
I wish to deeply thank Peter Sarnak, Vesselin Dimitrov and Elon Lindenstrauss for very fruitful discussions about this topic. I am extremely grateful to Valentin Blomer for helpful and insightful comments regarding a previous version of this manuscript. I wish to thank the referees for carefully reading the manuscript, their suggestions have notably improved the presentation. This work has been supported by the National Science Foundation under Grant No. DMS-1946333.

\section{The growth of the Epstein zeta function in the cusp}\label{sec:Epstein-cusp}
\begin{defi}
Define the degenerate spherical Eisenstein series with complex parameter $s$ and $g\in\mathbf{GL}_n(\mathbb{R})$
\begin{equation*}
E\left(g,s\right)=\frac{1}{2}|\det g|^{s/n} \sum_{0\neq \mathrm{v}\in \mathbb{Z}^n} \|\mathrm{v} g \|_2^{-s}\;.
\end{equation*}
\end{defi}
This function coincides with the Epstein zeta function\footnote{Our normalization for the Epstein zeta function is different from \cite{Terras} where $Z(g\cdot \tensor[^t]{g}{},\rho)=|\det g|^{-2\rho/n}E(g,2\rho)$.} of the quadratic form with Gram matrix $g\cdot \tensor[^t]{g}{}$. The series converges absolutely for $\Re s>n$ and can be analytically continued to a meromorphic function of $s\in \mathbb{C}$. The unique pole of $E\left(g,s\right)$ is at $s=n$. This pole is simple with residue\footnote{This can be deduced from the fact that the number of lattice points in a sphere of radius $R$ is asymptotic to the volume of the sphere as $R\to \infty$ for any unimodular lattice.}
\begin{equation*}
\Res_{s=n}E\left(g,s\right)=
\frac{\pi^{n/2}}{\Gamma(n/2)}\;.
\end{equation*}
The constant on the right is half the surface area of the $(n-1)$-dimensional unit sphere. Notice that it does not depend on $g$.

We have normalized $E(g,s)$ using the determinant to make it a well-defined function on $\lfaktor{\mathbf{PGL}_n(\mathbb{Z})}{\mathbf{PGL}_n(\mathbb{R})}$. 

\begin{defi}
We will also make use of a completed version of $E(g,s)$ defined as
\begin{equation*}
E^*(g,s)=\pi^{-s/2}\Gamma\left(\frac{s}{2}\right) E(g,s)\;.
\end{equation*}
\end{defi}

The functional equation \cite{Epstein}, cf.\ \cite[Proposition 10.2]{ELMVCubic}, is especially simple for the completed Eisenstein series
\begin{equation*}
E^*(g,n-s)=E^*(\tensor[^t]{g}{^{-1}},s)
\end{equation*}
It follows that $E^*(g,s)$ is holomorphic in $s$ except for two simple poles at $s=0,n$ with residues $-1,1$ respectively.

Our first goal is to understand the behavior of this function high in the cusp.
The following theorem due to Riemann for $n=1$ and  Terras \cite{Terras} for $n>1$ is a variant of the approximate functional equation for the Epstein zeta function. We provide a proof using Mellin inversion.
\begin{thm}\label{thm:approx-func-eq}
For any $0,n\neq s\in\mathbb{C}$  and $g\in\mathbf{GL}_n(\mathbb{R})$
\begin{equation*}
E^*(g,s)=-\frac{1}{s}-\frac{1}{n-s}+\frac{1}{2}\sum_{0\neq v \in \mathbb{Z}^n} f\left(s, \frac{\|vg\|_2}{|\det g|^{1/n}}\right)
+\frac{1}{2}\sum_{0\neq v \in \mathbb{Z}^n} f\left(n-s, \frac{\|v\tensor[^t]{g}{^{-1}}\|_2}{|\det \tensor[^t]{g}{^{-1}}|^{-1/n}}\right)
\end{equation*}
where 
\begin{equation*}
f(s,a)\coloneqq (\pi a^2)^{-s/2}\Gamma\left(\frac{s}{2},\pi a^2\right)
=\int_1^\infty t^{s/2}\exp(-\pi t a^2) \dif^\times t\;.
\end{equation*}
\end{thm}
\begin{proof}
The Mellin transform in the $a$ variable of $f(s,a)$ is exactly $\pi^{-\sigma/2}\Gamma\left(\frac{\sigma}{2}\right)(\sigma-s)^{-1}$. Because of the exponential decay of the Gamma function in the vertical direction we can use Mellin inversion to write
\begin{equation*}
\frac{1}{2}\sum_{0\neq v \in \mathbb{Z}^n} f\left(s, \frac{\|vg\|_2}{|\det g|^{1/n}}\right)=\frac{1}{2\pi i}\int_{\Re \sigma=n+\delta} \frac{E^*(g,\sigma)}{\sigma-s} \dif \sigma
\end{equation*}
for any $1>\delta>0$. We shift the contour of integration to the line $\Re \sigma=-\delta$ collecting residues at the $\sigma=0,s,n$. To justify the contour shift we claim that $E^*(g,s)$ decays exponentially in the vertical direction uniformly in any interval $a\leq \Re s \leq b$. To the right of the critical strip this follows from the definition using lattice summation and the exponential decay in the vertical direction of the Gamma function. To the left of the critical strip this can be deduced using the functional equation and inside the critical strip using the Phragm\'en-Lindel\"of principle.

The residue at $\sigma=s$ coincides with $E^*(g,s)$ and the other two residues produce the terms $-\frac{1}{s}$ and $-\frac{1}{n-s}$ in the claim. The proof is concluded by applying the functional equation and the change of variable $\sigma \mapsto n-\sigma$:
\begin{equation*}
-\frac{1}{2\pi i}\int_{\Re \sigma=-\delta} \frac{E^*(g,\sigma)}{\sigma-s} \dif \sigma
=-\frac{1}{2\pi i}\int_{\Re \sigma=-\delta} \frac{E^*(\tensor[^t]{g}{^{-1}},n-\sigma)}{\sigma-s} \dif \sigma
=\frac{1}{2\pi i}\int_{\Re \sigma=n+\delta} \frac{E^*(\tensor[^t]{g}{^{-1}},\sigma)}{\sigma-(n-s)} \dif \sigma\;.
\end{equation*}
The latter integral is equal to the dual sum because of Mellin inversion.
\end{proof}

\begin{cor}\label{cor:Epstein-pos}
For any real $s\neq 0,n$ the function $E^*(g,s)$ is real and satisfies
\begin{equation*}
E^*(g,s)\geq -\frac{1}{s}-\frac{1}{n-s}\;.
\end{equation*}
\end{cor}

\begin{defi}
For any $g\in\mathbf{GL}_n(\mathbb{R})$ set
\begin{equation*}
\lambda_1(g)\coloneqq |\det g|^{-1/n} \min_{0\neq \mathrm{v} \in \mathbb{Z}^n} \|\mathrm{v} g\|_2
\end{equation*}
to be the length of the shortest non-trivial vector in the unimodular lattice homothetic to $\mathbb{Z}^n g$. 
\end{defi}
Recall that Mahler's compactness criterion implies that $\lambda_1\colon \lfaktor{\mathbf{PGL}_n(\mathbb{Z})}{\mathbf{PGL}_n(\mathbb{R})}\to \mathbb{R}_{>0}$ is a proper continuous function.

\begin{thm}\label{thm:Epstein-cusp}
Denote by $V_{n-1}=\nicefrac{\pi^{(n-1)/2}}{\Gamma\left((n+1)/2\right)}$ the volume of the $(n-1)$-dimensional unit ball.
For any real $0<s<n$ if 
\begin{equation*}
\lambda_1(g)\leq \begin{cases}
2^{-1/(s-1)} & s>1\\
1 & s =1\\
V_{n-1} 2^{-\frac{(n-s)(n-1)}{n-s-1}} & s<1
\end{cases}
\end{equation*}
then 
\begin{equation*}
E^*(g,s)+\frac{1}{s}+\frac{1}{n-s}\gg
\begin{cases}
\lambda_1(g)^{-s} & s>1\\
-\lambda_1(g)^{-1}\log \lambda_1(g) & s=1\\
\left(\lambda_1(g) \frac{2^{n-1}}{V_{n-1}}\right)^{-(n-s)/(n-1)} & s<1
\end{cases}
\end{equation*}
The implied constant above is independent of all parameters.
\end{thm}
\begin{remark}
The lower bound above is optimal up to a constant. This can be seen by applying iteratively the Fourier expansion of $E^*(g,s)$ due to Terras \cite{TerrasFourier} to compute the constant term of $E^*(g,s)$. Write the Iwasawa decomposition of $g=u\cdot a \cdot k$ where $a=\diag(y_1,\ldots,y_n)$ with positive entries, $u\in \mathbf{N}(\mathbb R)$ is lower triangular unipotent and $k\in\mathbf{O}_n(\mathbb R)$.
The constant term of $E^*(g,s)$ is equal to\footnote{The original expansion in \cite{TerrasFourier} is in terms of the Iwasawa decomposition of $\tensor[^t]{g}{^{-1}}$. To pass to an expression in terms of the decomposition of $g$ we apply first the functional equation.}
\begin{equation*}
\sum_{k=0}^{n-1} \zeta^*(s-k) \prod_{i=1}^{n-k-1} \left(\frac{y_{i+1}}{y_i}\right)^{i (1-s/n)} \cdot 
\prod_{i=n-k}^{n-1} \left(\frac{y_{i+1}}{y_i}\right)^{(n-i) s/n}\;.
\end{equation*}
While $\lambda_1(g) \asymp_n \prod_{i=1}^{n-1} \left(\frac{y_{i+1}}{y_i}\right)^{-(n-i)/n}$. Combining these two expressions we deduce that at least the constant term is asymptotic to the lower bound in the theorem above. The difficulty in establishing the theorem using the Fourier expansion is that it is hard to analyze for $n>2$ the contribution of the non-constant terms in the Fourier expansion when $\diag(y_1,\ldots,y_n)$ is near the walls of the positive Weyl chamber. Instead we study the behavior in the critical strip using an approximate functional equation. 
\end{remark}
\begin{proof}
Assume first $s\geq 1$ then $\lambda_1(g)\leq 1$ by assumption.
Because the sums  over the lattices $\mathbb{Z}^n g$ and $\mathbb{Z}^n \tensor[^t]{g}{^{-1}}$ in Theorem \ref{thm:approx-func-eq} are positive, we can compute a lower bound by restricting the sum to a line going through a vector of minimal length in $\mathbb{Z}^n g$. This implies   
\begin{equation*}
E^*(g,s)+\frac{1}{s}+\frac{1}{n-s}\geq \frac{1}{2} \sum_{0\neq b\in\mathbb{Z}} f(s,|b|\lambda_1(g))\;.
\end{equation*}
The integral representation of $f(s,a)$ implies that it is a monotonic decreasing function of $a$ for $a>0$. Hence the right hand side above can be bounded below by
\begin{equation*}
\frac{1}{\lambda_1(g)}\int_{\lambda_1(g)}^{\infty} f(s,a)\dif a=\frac{\lambda_1(g)^{-s}}{2}
\int_{\lambda_1(g)^2}^{\infty} t^{(s-1)/2}\operatorname{erfc}(\sqrt{\pi t})\dif^\times t
\geq \frac{\lambda_1(g)^{-s}}{2} \int_{\lambda_1(g)^2}^{1} t^{(s-1)/2}\operatorname{erfc}(\sqrt{\pi t})\dif^\times t \;.
\end{equation*}
The equality above follows by applying Fubini to the integral representation of $f(s,a)$ and the change of variables $t\mapsto \lambda_1(g)^2t$.
We bound the latter integral using the monotonicity inequality $\operatorname{erfc}(x)\geq\operatorname{erfc}(\sqrt{\pi})$ for $0\leq x\leq \sqrt{\pi}$: 
\begin{equation}\label{eq:lambda_1_int_small}
\frac{\lambda_1(g)^{-s}}{2}
\int_{\lambda_1(g)^2}^{1} t^{(s-1)/2}\operatorname{erfc}(\sqrt{\pi t})\dif^\times t
\gg \frac{\lambda_1(g)^{-s}}{2} \int_{\lambda_1(g)^2}^{1} t^{(s-1)/2}\dif^\times t=
\begin{cases}
\frac{\lambda_1(g)^{-s}-\lambda_1(g)^{-1}}{s-1} & s\neq 1\\
-\lambda_1(g)^{-1}\log \lambda_1(g) & s=1
\end{cases}
\end{equation}

This establishes the claim in case $s=1$. In case $s>1$  the assumption $\lambda_1(g)^{s-1}<1/2$ implies that $\lambda_1(g)^{-1}\leq1/2 \lambda_1(g)^{-s}$ and the claim follows again from \eqref{eq:lambda_1_int_small}.

The lower bound for $s<1$ will follow from applying the $s>1$ case to the dual lattice which also contributes to $E^*(g,s)$ with $s$ replaced by $n-s$. We need only to establish
\begin{equation}\label{eq:lambda1-dual-ineq}
\lambda_1\left(\tensor[^t]{g}{^{-1}}\right)^{n-1}\leq \frac{2^{n-1}}{V_{n-1}} \lambda_1(g) \;.
\end{equation}
To prove inequality \eqref{eq:lambda1-dual-ineq} fix  $v_1,\ldots,v_n$ a basis of the lattice $\Lambda\coloneqq\mathbb{Z}^n g$, where $v_1$ is a vector of minimal length. Denote by $\mathrm{v}_1^*,\ldots,\mathrm{v}_n^*$ the dual basis of $\Lambda^*\coloneqq\mathbb{Z}^n \tensor[^t]{g}{^{-1}}$. Then $\mathrm{v}_2^*,\ldots,\mathrm{v}_n^*$ span a lattice $\Lambda_1^*$ in the $n-1$-dimensional hyperplane $\mathrm{v}_1^\perp$ and
\begin{equation*}
|\det g|^{-1}=\operatorname{covol}\left(\Lambda^*\right)=\operatorname{covol}\left(\mathrm{v}_1^*\mathbb{Z}+\Lambda_1^*\right)
=\left|\left\langle \mathrm{v}_1^*, \frac{\mathrm{v}_1}{\|\mathrm{v}_1\|}\right\rangle\right| \operatorname{covol}\left(\Lambda_1^*\right)= \|\mathrm{v}_1\|^{-1} \operatorname{covol}\left(\Lambda_1^*\right)\;.
\end{equation*}
Hence $\operatorname{covol}\left(\Lambda_1^*\right)= \lambda_1(g) |\det g|^{1/n-1}$ and Minkowski's first theorem implies that there is a vector $\mathrm{v}^*\in \Lambda_1^*\subset \Lambda^*$ satisfying $V_{n-1}\|\mathrm{v}_*\|_2^{n-1}\leq 2^{n-1}  \lambda_1(g) |\det g|^{1/n-1}$. This implies \eqref{eq:lambda1-dual-ineq} and the second claimed inequality.
\end{proof}

\begin{cor}\label{cor:Epstein-cusp-uniform}
Assume $1/n\leq s<1$. There are effectively computable constants $A_0,B_0>0$, depending only on $n$ and $s$, such that for all $g\in \mathbf{GL}_n(\mathbb{R})$
\begin{equation*}
E^*(g,ns)\geq A_0 \lambda_1(g)^{-ns}-B_0\;.
\end{equation*}
In fact,
\begin{align*}
A_0&=\operatorname{erfc}(\sqrt{\pi})
\begin{cases}
\frac{1}{2(ns-1)} & s>1/n\\
\log 2 &s=1/n
\end{cases},
&
B_0&=\frac{1}{n}\left(\frac{1}{s}+\frac{1}{1-s}\right)+A_0
\begin{cases}
2^{ns/(ns-1)} & s>1/n\\
2 & s=1/n
\end{cases}
\end{align*}
are admissible.
\end{cor}
\begin{proof}
This follows immediately with $B_0=\frac{1}{n}\left(\frac{1}{s}+\frac{1}{1-s}\right)$ from Theorem \ref{thm:Epstein-cusp} above if $\lambda_1(g)< 2^{-1/(ns-1)}$ and $s>1/n$ or if $s=1/n$ and $\lambda_1(g)<1/2$. The specific value of $A_0$ is a direct consequence of the proof.

Otherwise, assume first that $s>1/n$. If $\lambda_1(g)\geq 2^{-1/(ns-1)}$ then $\lambda_1(g)^{-ns}\leq 2^{ns/(ns-1)}$. Moreover, 
we know from Corollary \ref{cor:Epstein-pos} that $E^*(g,ns)\geq - \frac{1}{n}\left(\frac{1}{s}+\frac{1}{1-s}\right)$. Hence the claim holds for any $g$ with $B_0=\frac{1}{n}\left(\frac{1}{s}+\frac{1}{1-s}\right)+2^{ns/(ns-1)}A_0$. The argument for $s=1/n$ is analogous.
\end{proof}

\section{Toral periods of the Epstein zeta function}\label{sec:periods}
We recall a formula originally due to Hecke that relates Hecke $L$-functions of number field to toral periods of the Epstein zeta function. The proofs are straightforward using the unfolding method and can be extended to any Grossencharakter $L$-function in the ad\`elic setting, cf. \cite[Lemma 10.4]{ELMVCubic} and \cite{Wielonsky}.

Let $E/\mathbb{Q}$ be a degree $n$ number field with $r_1$ real places and $r_2$ inequivalent complex places. Denote by $D\coloneqq\disc(\mathcal{O}_E)$ the discriminant of its ring integers and let $R\coloneqq \Reg(\mathcal{O}_E)$ be its regulator. Set $h\coloneqq\#\Cl(E)$.
Let $E_\infty$ be the \'etale-algebra  $E\otimes \mathbb{R}$  over $\mathbb{R}$.

Fix once and for all a ring isomorphism
\begin{equation*}
\iota\colon E_\infty\to \mathbb{R}^{r_1}\times \mathbb{C}^{r_2}
\end{equation*}
This map is unique up to post-composition with permutation of the real and complex places respectively and complex conjugation at each complex places. We henceforth identify the right hand side with $\mathbb{R}^n=\mathbb{R}^{r_1+2r_2}$ in the standard manner. For any $\mathbb{Z}$-lattice $\Lambda \subset E$ we denote by $\tensor[^\iota]{\Lambda}{}$ the element of $\lfaktor{\mathbf{PGL}_n(\mathbb{Z})}{\mathbf{PGL}_n(\mathbb{Z})}$ corresponding to the lattice $\iota(\Lambda)$. Specifically, the raw matrix of every $\mathbb{Z}$-basis of the lattice $\iota(\Lambda)$ is a representative of the coset $\tensor[^\iota]{\Lambda}{}$.

\begin{defi} We denote by $\left(\mathbb{R}^\times\right)^\Delta$ the diagonal embedding of $\mathbb{R}^\times$ in $E_\infty^\times$.
Set
\begin{equation*}
H\coloneqq \lfaktor{\left(\mathbb{R}^\times\right)^\Delta}{E_\infty^\times}=\lfaktor{\left(\mathbb{R}^\times\right)^\Delta}{\left(\mathbb{R}^\times\right)^{r_1}\times {\left(\mathbb{C}^\times\right)^{r_2}}}\,.
\end{equation*}
We identify $H$ with a maximal torus subgroup in $\mathbf{PGL}_n(\mathbb{R})$ using the map $\iota$.
The Haar measure on $H$ is normalized to be consistent with the standard Haar measures on $E_\infty^\times$ and $\mathbb{R}^\times$.
\end{defi}

Dirichlet's unit theorem implies that $\mathcal{O}_E^\times\slash \mathbb{Z}^\times$ is a lattice in $H$ of covolume 
\begin{equation*}
\frac{2^{r_1} \pi^{r_2} n R}{w}\;,
\end{equation*}
where $w$ is the number of roots of unity in $E$. 

\begin{defi}
Define $[H]=\lfaktor{\mathcal{O}_E^\times}{H}$ and normalize the Haar measure on $[H]$ so it has volume $1$. If $\dif ^\times h$ is the Haar measure on $H$, then the measure on $[H]$ descents from the Haar measure
\begin{equation*}
\frac{w\dif^\times h}{2^{r_1}\pi^{r_2}n R} \;.
\end{equation*}

If $\Lambda\subset E$ is a fractional $\mathcal{O}_E$-ideal then the stabilizer in $H$ of $\tensor[^\iota]{\Lambda}{}$ is the lattice $\mathcal{O}_E^\times\slash \mathbb{Z}^\times$. Hence $\tensor[^\iota]{\Lambda}{}H$ is a periodic $H$-orbit isomorphic to $[H]$. This orbit depends only on the ideal class of $\Lambda$. The ideal classes of $\mathcal{O}_E$ give rise to a finite collection of periodic $H$-orbits, c.f.\ \cite{ELMVPeriodic}. This collection is called a \emph{packet} of periodic $H$-orbits. 
\end{defi}

\begin{defi}\label{defi:partial-dedekind}
Let $\Lambda\subset E$ be a fractional $\mathcal{O}_E$-ideal. The partial Dedekind zeta-function of $\Lambda$ is defined by the Dirichlet series
\begin{equation*}
\zeta_{\Lambda}(s)\coloneqq\Nr(\Lambda)^s \sum_{0\neq \mathrm{v}\in \Lambda} \left|\Nr \mathrm{v} \right|^{-s}
\end{equation*}
that converges for $\Re s>1$.
The zeta function depends only the class of $\Lambda$ modulo the principal ideals. For every class group character $\chi\colon \Cl(E)\to\mathbb{C}^\times$ the class group $L$-function satisfies
\begin{equation*}
L(s,\chi)=\sum_{[\Lambda]\in \Cl(E)} \zeta_{\Lambda}(s) \bar{\chi}(\Lambda)\;. 
\end{equation*} 

We also define the completed partial zeta function as
\begin{equation*}
\zeta_{\Lambda}^*(s)\coloneqq \left(\pi^{-s/2}\Gamma\left(\frac{s}{2}\right)\right)^{r_1} \big((2\pi)^{-s}\Gamma\left(s\right)\big)^{r_2} |D|^{s/2} \zeta_{\Lambda}(s)\;.
\end{equation*}
The completed $L$-function $L^*(s,\chi)$ of a class group character $\chi$ is defined similarly. These satisfy a functional equation due to Hecke. The functional equation for these $L$-functions is a direct consequence of the functional equation for the completed Epstein zeta function and Theorem \ref{thm:Hecke-period} below. 
\end{defi}

\begin{thm}\label{thm:Hecke-period}[Hecke]
Let $\Lambda\subset E$ be a fractional $\mathcal{O}_E$-ideal. Then $\tensor[^\iota]{\Lambda}{}$ is a periodic $H$-orbit and for any $s\neq 0,1$ and
\begin{equation*}
\int_{[H]} E^*(\tensor[^\iota]{\Lambda}{}h,ns) \dif^\times h 
=\frac{w}{2^{r_1} n R} \zeta^*_{\Lambda}(s)\;,
\end{equation*}
\end{thm}
We reproduce the proof for completeness sake using the following important lemma, also due to Hecke. The crux of the proof is that the ring of $E_\infty^\times$-invariant polynomials on $E_\infty$ is generated by the norm function.
\begin{lem}\label{lem:Hecke-trick} [Hecke's Trick]
Equip $E_\infty$ with an Euclidean inner-product by summing the standard inner-products on each copy of $\mathbb{R}$ and $\mathbb{C}$. Then for all $\mathrm{v}\in E_\infty^\times$ 
\begin{equation*}
\int_{H} \left\|\mathrm{v}h\right\|_2^{-ns} \left|\Nr h\right|^{s}\dif^\times h=\frac{\pi^{r_2}\Gamma\left(s/2\right)^{r_1}\Gamma\left(s\right)^{r_2}}{\Gamma\left(ns/2\right)} |\Nr \mathrm{v}|^{-s}\;.
\end{equation*}

\begin{proof}
Using the change of variables $h\mapsto \mathrm{v}h$ we see that
\begin{equation*}
\int_{H} \left\|\mathrm{v}h\right\|_2^{-ns} \left|\Nr h\right|^s \dif^\times h=|\Nr \mathrm{v}|^{-s}
\int_{H} \ \left\|  h\right\|_2^{-ns} \left|\Nr h\right|^s \dif^\times h=|\Nr \mathrm{v}|^{-s}
I(s)\;.
\end{equation*}
To evaluate the integral on the right-hand side, $I(s)$, we calculate the integral 
$\int_{E_\infty^\times}e^{-\|y\|_2^2} \left|\Nr y\right|^s  \dif^\times y$
in two different ways. On one hand we use the compatibility of Haar measures on quotients and the change of variables $t\mapsto t \|y\|_2$
\begin{align*}
\int_{E_\infty^\times}e^{-\|y\|_2^2} \left|\Nr y\right|^{s} \dif^\times y&=\int_{H={\left(\mathbb{R}^\times\right)^\Delta}\backslash{E_\infty^\times}} \int_{\mathbb{R}^\times} e^{-(|t|\|y\|_2)^2} |t|^{ns} \left|\Nr y\right|^{s}\dif^\times t \dif^\times \left(\mathbb{R}^\times y\right)\\
&=\int_{H} \|y\|_2^{-ns} \left|\Nr y\right|^{s} \int_{\mathbb{R}^\times} e^{-(|t|\|y\|_2)^2} (|t|\|y\|_2)^{ns} \dif^\times t \dif^\times \left(\mathbb{R}^\times y\right)=I(s)\cdot \int_{0}^{\infty} e^{-t^2}t^{ns}\frac{2\dif t}{t}\\
&=I(s)\Gamma(ns/2)\;.
\end{align*}

On the other hand, using polar coordinates for each complex coordinate, the integral over $E_\infty^\times$ decomposes as a product
\begin{equation*}
\int_{E_\infty^\times}e^{-\|y\|_2^2} \left|\Nr y\right|^{s} \dif^\times y= \prod_{i=1}^{r_1} \int_{\mathbb{R}^\times}
e^{-t_i^2}|t_i|^s \dif^\times t_i \cdot \prod_{i=r_1+1}^{r_1+r_2} 2 \pi \int_{\mathbb{R}_{>0}}
e^{-r_i^2}r_i^{2s} \dif^\times r_i=\Gamma(s/2)^{r_1} \cdot \pi^{r_2}\Gamma(s)^{r_2}\;.
\end{equation*}
The proof concludes by comparing the two expressions for $\int_{E_\infty^\times}e^{-\|y\|_2^2} |y|^s  \dif^\times y$.
\end{proof}

\begin{proof}[Proof of Theorem \ref{thm:Hecke-period}]
We consider $h\in E_\infty^\times$ as an element of $\mathbf{GL}_n(\mathbb{R})$ using the map $\iota$, then $|\det h|=\left|\Nr h\right|$.
Rewrite the period of the Epstein zeta-function using the standard unwinding transformation
\begin{align*}
\int_{[H]} E(\tensor[^\iota]{\Lambda}{}h,ns)\dif^\times h = \frac{1}{2}&\operatorname{covol}(\Lambda)^s \sum_{0\neq \mathrm{v}\in \Lambda\slash \mathcal{O}_E^\times} \sum_{u\in \mathcal{O}_E^\times}
\int_{[H]} \|\mathrm{v}uh\|_2^{-sn} \left|\Nr h\right|^s\dif^\times h\\
=&\operatorname{covol}(\Lambda)^s \sum_{0\neq \mathrm{v}\in \Lambda\slash \mathcal{O}_E^\times} 
\int_H \|\mathrm{v}h\|_2^{-sn}\left|\Nr h\right|^s  \frac{w\dif^\times h}{2^{r_1}\pi^{r_2}n R} \;.
\end{align*}
The factor $1/2$ is absorbed in the difference between $\mathcal{O}_E^\times$ and the group $\mathcal{O}_E^\times \slash \mathbb{Z}^\times$.
The proof concludes by applying Lemma \ref{lem:Hecke-trick} and using the formula $\operatorname{covol}(\Lambda)=2^{-r_2}\sqrt{|D|} \Nr(\Lambda)$.
\end{proof}
\end{lem}

\section{The top period}\label{sec:highest}
We continue to fix a degree $n$ number field $E/\mathbb{Q}$ and carry all the notations from the previous sections. Henceforth we fix $1/2\leq s<1$ and define the function $Z\colon \Cl(E)\to\mathbb{C}$ as in the introduction
\begin{equation*}
Z(\Lambda)\coloneqq
\frac{w}{2^{r_1} n R} \zeta^*_{\Lambda}(s)=\int_{[H]} E^*(\tensor[^\iota]{\Lambda}{}h,ns) \dif^\times h\;.
\end{equation*}
Notice that the Fourier coefficients of the function $Z$ satisfy
\begin{equation*}
\hat{Z}(\chi)=\frac{1}{h}\sum_{[\Lambda]\in \Cl(E)} Z(\Lambda)\bar{\chi}(\Lambda)=
\frac{w}{2^{r_1} n h R} L^*(s,\chi)\;.
\end{equation*}

In this section we prove the following lower bound on the value of $Z$ at the identity class. This is a key part of our argument. 
\begin{prop}\label{prop:Z(O)-bound}
Let $1/2\leq s<1$. There are effectively computable constants $A_1,B_0>0$ depending only on $s$ and $n$ such that 
\begin{equation*}
Z(\mathcal{O}_E) \geq \frac{A_1|D|^{s/2}-B_0 R}{R}
\end{equation*}
In particular, $Z(\mathcal{O}_E)$ is positive if $|D|^{s/2}/R\gg_{s,n} 1$.
\end{prop}
We observe that the lower bound depends only on the regulator and not on the shape of the lattice of roots of unity. This is possible because the exponential map converts a linear combination of trace-less vectors in the logarithmic space to a product of units in $E_\infty$. Hence the average length in $E_\infty$, over a fundamental domain of units in the logarithmic space, almost decomposes as a product of averages. This reduces the proposition to a question of bounding a product of lengths of vectors forming a basis for the unit lattice. To control the latter we approximate the unit lattice by the sub-lattice spanned by vectors realizing the successive minima and use Minkowski's second theorem.
\begin{proof}
Consider as usual the logarithmic group homomorphism $\log_E\colon  E_\infty^\times \to \mathbb{R}^{r_1+r_2}$
\begin{equation*}
\log_E(u_1,\ldots,u_{r_1+r_2})\coloneqq\left(\log |u_1|_{\mathbb{R}},\ldots,\log|u_{r_1}|_{\mathbb{R}},
2\log|u_{r_1+1}|_{\mathbb{C}},\ldots,2\log|u_{r_1+r_2}|_{\mathbb{C}} \right)\;.
\end{equation*}
We will need the right-inverse 
\begin{equation*}
\exp_E(x_1,\ldots,x_{r_1+r_2})\coloneqq (\exp(x_1),\ldots,\exp(x_{r_1}),\exp(x_{r_1+1}/2),\ldots,\exp(x_{r_1+r_2}/2))\;.
\end{equation*}
The kernel of $\log_E$ is the group of elements whose coordinate-wise absolute values are all $1$. It is a compact subgroup that acts on $E_\infty$ by orthogonal transformation. In particular, $E^*(g,s)$ is invariant under right multiplication by $\ker (\log_E)$.
 
The map $\log_E$ furnishes a homomorphism from $H$ onto the trace $0$ subspace $\mathbb{R}^{r_1+r_2}_0$.  Dirichlet's units theorem states that the image of $\mathcal{O}_E^\times$ is a lattice in $\mathbb{R}^{r_1+r_2}_0$ of covolume $R$, where the covolume is computed with respect to the usual inner product on $\mathbb{R}^{r_1+r_2}$. Hence we can compute the integral of the spherical Epstein zeta function over the periodic $H$-orbit $\tensor[^\iota]{\mathcal{O}}{_E}H$ using a normalized Lebesgue measure on $\mathbb{R}^{r_1+r_2}_0$.
\begin{equation*}
\int_{[H]} E^*(\tensor[^\iota]{\mathcal{O}}{_E}h,ns)\dif^\times h= 
\frac{1}{R}\int_{\mathcal{F}} E^*(\tensor[^\iota]{\mathcal{O}}{_E}\exp_E(x_1,\ldots,x_{r_1+r_2-1}),ns) \dif^0(x_1,\ldots x_{r_1+r_2})\;,
\end{equation*}
where $\mathcal{F}$ is any fundamental domain for $\log_E(\mathcal{O}_E^\times)$ in $\mathbb{R}^{r_1+r_2}_0$ and $\dif^0(x_1,\ldots, x_{r_1+r_2})$ is the standard Lebesgue measure on the trace-less subspace $\mathbb{R}^{r_1+r_2}_0$. Corollary \ref{cor:Epstein-cusp-uniform} and the formula above imply that
\begin{equation}\label{eq:Z(OE)-cusp}
Z(\mathcal{O}_E) \geq  A_0 \frac{1}{R}
\int_{\mathcal{F}} 
\lambda_1(\tensor[^\iota]{\mathcal{O}}{_E}\exp_E(x))^{-ns} \dif^0 x-B_0
=A_0 I_\lambda(s)-B_0\;.
\end{equation}
Our aim now is to provide a proper lower bound for the normalized integral $I_\lambda(s)$.
Denote by $\|x\|_{\infty}=\max(|x_1|,\ldots,|x_{r_1}|,|x_{r_1+1}|/2,\ldots,|x_{r_1+r_2}|/2)$ the supremum norm on $\mathbb{R}^{r_1+r_2}$. This restricts to a norm on $\mathbb{R}^{r_1+r_2}_0$. Denote by $\tilde{V}_{r_1,r_2}$ the Lebesgue measure of the unit ball of the latter norm. Let $\theta_1,\ldots,\theta_{r_1+r_2-1}\in\log_E(\mathcal{O}_E^\times)$ be vectors realizing the successive minima of the lattice $\log_E(\mathcal{O}_E^\times)$ with respect to the $\|\bullet\|_\infty$ norm. By Minkowski's second theorem
\begin{equation}\label{eq:Minkowski-Theta}
\|\theta_1\|_\infty \cdots \|\theta_{r_1+r_2-1}\|_\infty \cdot \tilde{V}_{r_1,r_2}\leq 2^{r_1+r_2-1} R
\end{equation}
Denote by $\Theta\subset \log_E(\mathcal{O}_E^\times) \subset \mathbb{R}^{r_1+r_2}_0$ the lattice spanned by $\theta_1,\ldots,\theta_{r_1+r_2-1}$.
Define $$\mathcal{F}_\Theta\coloneq \left\{\sum_{j=1}^{r_1+r_2-1} \varepsilon_j \theta_j \mid 0\leq \varepsilon_j < 1 \right\}\;.$$
It is a fundamental domain for $\Theta$. This domain can be covered by exactly $\left[ \log_E(\mathcal{O}_E^\times) \colon \Theta\right]$ fundamental domains of $\log_E(\mathcal{O}_E^\times)$. We can now evaluate \eqref{eq:Z(OE)-cusp} over each of these domains to deduce
\begin{align*}
I_\lambda(s) &= \frac{1}{\operatorname{covol}{\Theta}}
\int_{\mathcal{F}_\Theta} 
\lambda_1(\tensor[^\iota]{\mathcal{O}}{_E}\exp_E(x))^{-ns} \dif^0x\\
&\gg_{s, n}\frac{|D|^{s/2}}{\operatorname{covol}{\Theta}} 
\int_{\mathcal{F}_\Theta} 
\|\exp_E(x)\|_2^{-ns} \dif^0x\;,
\end{align*}
where in the last inequality we have used the fact the the covolume of $\mathcal{O}_E$ is $2^{-r_2}\sqrt{|D|}$ and that it contains the short vector $(1,\ldots,1)$. 

Define the norm $\|y\|_{E_\infty}=\max(|y_1|_{\mathbb{R}},\ldots,|y_{r_1}|_{\mathbb{R}},|y_{r_1+1}|_{\mathbb{C}},\ldots,|y_{r_1+r_2}|_{\mathbb{C}})$.
We apply the inequality $\|\bullet\|_2\leq \sqrt{r_1+r_2} \|\bullet \|_{E_\infty}$ in $E_\infty$ and then rewrite the last integral using the basis $\theta_1,\ldots\theta_{r_1+r_2-1}$.
\begin{align}
\nonumber
\frac{1}{\operatorname{covol}{\Theta}}\int_{\mathcal{F}_\Theta}  &\|\exp_E(x)\|_2^{-ns} \dif^0x
\gg_{n,s} \frac{1}{\operatorname{covol}{\Theta}} \int_{\mathcal{F}_\Theta}  \|\exp_E(x)\|_{E_\infty}^{-ns} \dif^0 x
\geq \frac{1}{\operatorname{covol}{\Theta}}\int_{\mathcal{F}_\Theta}  \exp(-ns\|x\|_\infty) \dif^0 x\\
\nonumber
&\geq\int_0^1 \cdots \int _0^1 \exp(-ns\sum_{j+1}^{r_1+r_2-1}\varepsilon_j\|\theta_j\|_\infty )
\dif \varepsilon_1\cdots \dif \varepsilon_{r_1+r_2-1}
= \prod_{j=1}^{r_1+r_2-1} \int_0^1 \exp(-n s \varepsilon_j\|\theta_j\|_\infty) \dif \varepsilon_j\\
\label{eq:Ftheta-integral}
&= \prod_{j=1}^{r_1+r_2-1} \frac{1-\exp(-n s \|\theta_j\|_\infty)}{n s \|\theta_j\|_\infty}\;,
\end{align}
where in the second line we have used the triangle inequality for the norm $\|\bullet\|_\infty$ on $\mathbb{R}^{r_1+r_2}_0$.
We bound the denominator using Minkowski's second theorem \eqref{eq:Minkowski-Theta}. To bound the numerator we use the inequality $\|\log_E(y)\|_\infty\gg_n 1$ for every $y\in\mathcal{O}_E^\times\setminus \mu_E$, where $\mu_E<E^\times$ is the group of roots of unity. This inequality with an effective constant follows from the Northcott property, cf.\ \cite[\S1.6.15]{BombieriGubler}. The best possible bound (up to a multiplicative constant) follows from the recent breakthrough of V. Dimitrov \cite{Dimitrov} resolving the Schinzel-Zassenhaus conjecture:
\begin{equation*}
\|\log_E(y)\|_\infty\geq \frac{\log 2}{4 n}\;.
\end{equation*}
Worse bounds follow  from the result of Dobrowolski \cite{Dobrowolski} towards the Lehmer conjecture. See also Blanskby-Montgomery \cite{BlanksbyMontgomery} and Stewart \cite{Stewart}.
The claim finally follows by applying these bounds for the numerator and denominator in \eqref{eq:Ftheta-integral} and substituting into \eqref{eq:Z(OE)-cusp}.
\end{proof}

\bibliographystyle{alpha}
\bibliography{non_vanishing_bib}
\end{document}